\documentclass[preprint,12pt]{article}

\usepackage{amssymb}
\usepackage{amsthm}
\usepackage{amsfonts}
\usepackage{amsmath}

\usepackage{epsfig}
\usepackage{graphicx}
\usepackage{color}

\newcommand{\CV}{\mathcal{C}(V)}

\newcommand{\G}{\mathcal{G}}
\renewcommand{\L}{\mathcal{L}}

\def\RR{\mathbb{R}}
\def\L{\mathcal{L}}

\renewcommand{\AA}{{\sf A}}
\newcommand{\BB}{{\sf B}}
\newcommand{\CC}{{\sf C}}

\newcommand{\HH}{{\sf H}}
\newcommand{\II}{{\sf I}}
\newcommand{\MM}{{\sf M}}
\newcommand{\Ss}{{\sf S}}

\newtheorem{lemma}{Lemma}
\newtheorem{prop}{Proposition}
\newtheorem{theorem}{Theorem}
\newtheorem{corollary}{Corollary}

\def\v{{\sf v}}\def\u{{\sf u}}\def\w{{\sf w}}
\def\s{{\sf s}}
\def\0{{\sf 0}}
\def\LL{{\sf L}}
\newcommand{\XX}{{\sf X}}

\begin{document}


%
%
%
%
%

\title{Group Inverse of the Laplacian of Connections of Networks}


\author{Silvia Gago}

\date{Departament de Matem\`atiques \\
Universitat Polit\`ecnica de Catalunya}

\maketitle

\begin{abstract}
In previous works the group inverse of a network obtained by some perturbations, as the deletion of a vertex, the addition of a new vertex, contraction of an edge, etc. is obtained in terms of the group inverse of the original network. In this work, two given networks are connected with some new edges and the group inverse of the new network is related with the group inverses of the two original networks. In particular, the formula for the connection of two networks by just one edge is obtained, and besides the formula for the Kirchhoff index of this kind of network.
\end{abstract}



\section{Introduction}
\label{int}
In the last two decades a great amount of different parameters have been proposed to quantify structural properties of a network, such as their connectivity or their robustness, see \cite{ESMJK11}. Between them, maybe the most useful, and hence the most studied, is the {\it effective resistance} that has associated the {\it total resistance} as a global parameter. This concept appears when the network is considered as an electrical circuit where each edge is understood as a resistance whose inverse is called the conductance of the edge. The analogy between network and electrical circuits has been widely used in the Organic Chemistry framework, where the total resistance is known as the {\it Kirchhoff index}. In fact, it was in this area where the effective resistance emerged  as a better alternative to other parameters used for discriminating among different molecules with similar shapes and structures, see \cite{KR93}. 

However, the determination of the effective resistance is, in general, a veritable {\it tour de force} from the computational point of view. Firstly, it is highly sensitive from the small perturbations on the network, such as increasing or decreasing the conductance of the edges, adding or deleting some vertex, see for instance \cite{CEM14,C11,YK13}. On the other hand, as a function, the effective resistance is closely related to the Green function of the network, or from the matrix point of view, with the {\it group inverse} of the combinatorial Laplacian of the network. Therefore, the bigger is the network, the more difficult is to compute the group inverse of the combinatorial Laplacian. So, a common strategy is to consider large complex networks as {\it composite networks},  and to find   relations between the effective resistance and Kirchhoff indices of the original networks and those of their composite networks; see for instance \cite{BCE12,YK13,ZY09} and references therein. 
This strategy made necessary to extend the concept of effective resistance, defining the effective resistance between any pair of vertices with respect to a nonnegative value and a weight on the vertex set, see \cite{BCEG10,BCE12}. From the operator point of view, this means that the effective resistance can be also considered associated with positive semi--definite Schr\"odinger operators, instead of only for Laplacian operators. From the matrix point of view, it result that not only we can define the effective resistance for symmetric $M$--matrices, but else that even the classical effective resistance for the Laplacian matrix; that is, a singular and weakly diagonal dominant $M$--matrix, of a ``big'' network can be computed in terms of the effective resistance for non--singular symmetric $M$--matrices that are not necessarily d.d. ones. 

In this paper we consider composite networks built {\it connecting} two simpler networks. This kind of composition encompasses many composite operations on networks as join, corona, cluster or adding vertices. Since the key tool in the determination of the effective resistances in these composite networks is to compute their Green function, we first prove a general result about the computation of the group inverse of a singular symmetric matrix that has $0$ as simple eigenvalue and is partitioned into blocks, in terms of the inverse of the diagonal blocks. Of course, there exists many results in the literature computing generalized inverse of matrices partitioned into blocks, see for instance \cite{HM75}. However, since their aim is to provide a very general formulae, they are usually intricate and very difficult to manage them in specific cases. So, in specific frameworks it is more useful to provide, more or less directly, the formula required in that particular context. This is the case for the combinatorial Laplacians of a complex network obtained by connecting ``smaller" ones. Since this matrix can be seen as partitioned into blocks, where the non diagonal blocks represent the conductances of the new connecting edges, our first aim is to obtain a very general, but simple, result in relation with the computation of the the group inverse of a singular symmetric matrix for which $0$ is a simple eigenvalue. The second section of this paper is devoted to achieve this goal.

\section{The group inverse of a symmetric matrix partitioned into blocks}
\label{}

Consider the symmetric matrix $$\AA=\left(\begin{array}{cc}
\HH_1 & {\sf B}\\
{\sf B}^{\top} & \HH_2
\end{array}
\right),$$ where ${\sf B}\in \mathcal{M}_{n\times m}$, $\HH_1\in \mathcal{M}_{n\times n}$ and $\HH_2\in \mathcal{M}_{m\times m}$. Therefore,  $\HH_1$ and $\HH_2$ are  symmetric matrices and moreover we assume that $\HH_2$ is also invertible. From the identity 
$$ \left(\begin{array}{cc}
\HH_1 & {\sf B}\\
{\sf B}^{\top} &\HH_2
\end{array}
\right)\left(\begin{array}{c}
\v\\ \w\end{array}\right)=\left(\begin{array}{c}
{\sf 0}\\
{\sf 0}\end{array}\right),$$ 
we obtain that $\HH_1\v=-{\sf B}\w$, $\w=-\HH_2^{-1}{\sf B}^{\top}\v$ and hence,
$\big({\sf H}_1-{\sf B}\,\HH_2^{-1}\,{\sf B}^{\top}\big)\v=\0.$ 
 
Under the above  conditions the {\it Schur complement of $\HH_2$ in $\AA$}, usually denoted by $\AA/\HH_2$, is defined as the matrix   
\begin{equation}
\label{Schur:def}
{\sf S}={\sf H}_1-{\sf B}\,\HH_2^{-1}\,{\sf B}^{\top}.
\end{equation}
Clearly $\Ss$ is a symmetric matrix and the above identities show that $\AA$ is invertible iff ${\sf S}$ is invertible  and moreover $0$ is a simple eigenvalue of $\AA$ iff it is a simple eigenvalue of $\Ss$. \vspace{.25cm}

When $\AA$ is invertible, an expression for its inverse in terms of the inverses of $\HH_2$ and its Schur complement is known:
 \begin{equation}
\label{Schur:inv}
\AA^{-1}=\left(\begin{array}{cc}
{\sf S}^{-1} & -{\sf S}^{-1}{\sf B}\HH_2^{-1}\\[1ex]
-\HH_2^{-1}{\sf B}^\top{\sf S}^{-1}&
\HH_2^{-1}(\HH_2+{\sf B}^\top{\sf S}^{-1}{\sf B})\HH_2^{-1}
\end{array}
\right).
\end{equation}
 
Analogously, when $\HH_1$ is invertible, the   Schur complement of $\HH_1$ in $\AA$, usually denoted by $\AA/\HH_1$, is given by  $\widehat {\sf S}=\HH_2-{\sf B}^\top\,\HH_1^{-1}\,{\sf B}$ and hence $\AA$ is invertible iff $\widehat {\sf S}$ also is and then we have a formula for $\AA^{-1}$ similar to \eqref{Schur:inv}. Moreover, when both $\HH_1$ and $\HH_2$ are invertible, from the identities 

$$\begin{array}{rl}
\BB\HH_2^{-1}\widehat {\sf S}=&\hspace{-.3cm}\BB-\BB\HH_2^{-1}{\sf B}^\top\,\HH_1^{-1}\,{\sf B}={\sf S}\HH_1^{-1}\,{\sf B}\\[1ex]
\big(\HH_2^{-1}+\HH_2^{-1}{\sf B}^\top{\sf S}^{-1}{\sf B}\HH_2^{-1}\big)\widehat {\sf S}=&\hspace{-.3cm} {\sf I}-\HH_2^{-1}{\sf B}^\top\HH_1^{-1}\,{\sf B}+\HH_2^{-1}{\sf B}^\top{\sf S}^{-1}{\sf B}\\[1ex]
-&\hspace{-.3cm}\HH_2^{-1}{\sf B}^\top{\sf S}^{-1}\big({\sf B}\HH_2^{-1}{\sf B}^\top\big)\HH_1^{-1}\,{\sf B}\\[1ex]
=&\hspace{-.3cm} {\sf I}-\HH_2^{-1}{\sf B}^\top\Big[\HH_1^{-1}\,{\sf B}-{\sf S}^{-1}{\sf B}+{\sf S}^{-1}\big(\HH_1-{\sf S}\big)\HH_1^{-1}\,{\sf B}\big]\\[1ex]
=&\hspace{-.25cm} {\sf I}-\HH_2^{-1}{\sf B}^\top\Big[\HH_1^{-1}\,{\sf B}-{\sf S}^{-1}{\sf B}+{\sf S}^{-1}{\sf B}-\HH_1^{-1}\,{\sf B}\big]=\II
\end{array}$$
we get that ${\sf S}^{-1}\BB\HH_2^{-1}=\HH_1^{-1}\,{\sf B}\widehat {\sf S}^{-1}$ and hence, 
 \begin{equation}
\label{Schur:inv2}
\AA^{-1}=\left(\begin{array}{cc}
{\sf S}^{-1} & -\HH_1^{-1}\,{\sf B}\widehat {\sf S}^{-1}\\[1ex]
-\HH_2^{-1}{\sf B}^\top{\sf S}^{-1}&\widehat {\sf S}^{-1}
\end{array}
\right),
\end{equation}
see \cite[Identity (0.100)]{BG03}. 

Under similar hypothesis on the blocks of matrix $\AA$, the scenario dramatically changes when $\AA$ and hence $\Ss$ or $\widehat \Ss$ are singular. Then, we can not merely replace  $\Ss^{-1}$ by the group inverse  (that coincides with the  Moore-Penrose inverse) in the Identity \eqref{Schur:inv}, or in the Identity \eqref{Schur:inv2}, to obtain the expresion of the group inverse of $\AA$. 

The  expression for the Moore--Penrose inverse for general block matrices can be found in \cite{HM75}, but  is quite intricate and   difficult to manage, even under additional hypotheses on the blocks. Our goal is to obtain a friendly expression of the group inverse of $\AA$ under the additional hypothesis that $0$ is a simple eigenvalue. Observe  that the combinatorial Laplacian of a connected network always satisfies this property and moreover it is positive semidefinite.\ For this class of matrices, its group inverse can be characterized as follows.

\begin{lemma}
\label{GI:ch}
Assume that $\AA$ is singular, symmetric and   such that $0$ is a  simple eigenvalue. Then, if  $\u$ is an unitary $0$-eigenvector of $\AA$ the group inverse  $\AA^\#$ is characterized as the unique  singular matrix satisfying 
$$\AA\AA^\#=\AA^\#\AA=\II-\u\u^\top.$$
Moreover, $\AA^\#$ is symmetric, $\AA^\#\u=\0$ and $0$ is a simple eigenvalue of $\AA^\#$ .
\end{lemma}
%
Now we are ready to prove the main result in this section. 

\begin{theorem} 
\label{lemmap}
Let $\HH_1\in \mathcal{M}_{n\times n}(\RR)$, $\HH_2\in \mathcal{M}_{m\times m}(\RR)$ be two    symmetric matrices, ${\sf B}\in \mathcal{M}_{n\times m}(\RR)$, and assume that the symmetric matrix  $\AA=\left(\begin{array}{cc}
\HH_1 & {\sf B}\\
{\sf B}^{\top} & \HH_2
\end{array}
\right)$ is singular and  has $0$ as a simple eigenvalue. 
If $\HH_2$ is non singular and  ${\sf S}={\sf H}_1-{\sf B}\,\HH_2^{-1}\,{\sf B}^{\top}$ is the Schur complement of $\HH_2$ on $\AA$, then  $\AA^\#=\left(\begin{array}{cc}
\XX_1 & {\sf C}\\
{\sf C}^{\top} & \XX_2
\end{array}
\right)$ where
$$\begin{array}{rl}
\XX_1=&\hspace{-.25cm}\big(\II_n+\v\,\w^{\top}\HH_2^{-1}\BB^{\top}\big)\Ss^{\#}\big(\II_n+\BB\HH_2^{-1}\w\,\v^{\top}\big)+\v\,\w^{\top}\HH_2^{-1}\w\,\v^{\top},\\[1.25ex]
\XX_2=&\hspace{-.25cm}(\II_m-\w\,\w^{\top})\Big(\HH_2^{-1}+\HH_2^{-1}\BB^\top\Ss^{\#}\BB\HH_2^{-1}\Big)\big(\II_m-\w\,\w^{\top}\big),\\[1.25ex]
\CC=&\hspace{-.25cm}-\Big(\v\,\w^{\top}+\big(\II_n+\v\,\w^{\top}\HH_2^{-1}\BB^\top\big)\Ss^{\#}\BB\Big)\HH_2^{-1} \big(\II_m-\w\,\w^{\top}\big).\\[1.25ex]
\end{array}$$
and  $\u^\top=(\v^\top,\w^\top)$ is an unitary $0$--eigenvector of $\AA$.
\end{theorem}
\proof First notice that $\Ss\v=\0$ and moreover  $\w=-\HH_2^{-1}\BB^\top\v$, which implies that $\v\not=\0$. Therefore, if $\alpha=(1-\w^\top \w)^{-1}=(\v^\top\v)^{-1}$,  from Lemma \ref{GI:ch} we get that 
$$\Ss^\#\Ss =\Ss\Ss^\# =\II_{n}-\alpha\,\v\,\v^\top.$$
Applying newly Lemma \ref{GI:ch}, if the matrix  $\left(
\begin{array}{cc}
\XX_1 & \CC\\
\CC^{\top}&\XX_2
\end{array}
\right)$, is the group inverse of $\AA$, then  $\XX_1$ and $\XX_2$ are symmetric, 
$(\v^\top,\w^\top) \left(
\begin{array}{cc}
\XX_1 & \CC\\
\CC^{\top}&\XX_2
\end{array}
\right)  =  (\0^\top,\0^\top)$ and 
$$\left(\begin{array}{cc}
\HH_1 & \BB\\
\BB^{\top} &\HH_2
\end{array}
\right)
\left(
\begin{array}{cc}
\XX_1 & \CC\\
\CC^{\top}&\XX_2
\end{array}
\right)= \II-{\sf u}\,{\sf u}^\top=
\left(
\begin{array}{cc}
\II_n- \v\,\v^{\top}&-\v\,\w^\top \\
-\w\,\v^{\top}& \II_m-\w\,\w^{\top}
\end{array}
\right). 
$$
In particular, the following identities must be satisfied:
$$\begin{array}{lrl}
(a)& \HH_1 \XX_1+\BB\CC^{\top}& \hspace{-.25cm} =\II_n- \v\,\v^{\top}, \\[1ex]
(b)& \BB^{\top}\XX_1+ \HH_2\CC^{\top}& \hspace{-.25cm} =-\w\,\v^{\top}, \\[1ex]
(c)&  \v^\top\XX_1  & \hspace{-.25cm} =-\w^\top \CC^\top, \\[1ex]
(d) & \BB^{\top}\CC+ \HH_2\XX_2 & \hspace{-.25cm} =\II_m- \w\,\w^{\top}.
\end{array}$$
From $(b)$ we get $\CC^{\top}=-\HH_2^{-1}\big(\w\,\v^{\top}+\BB^{\top}\XX_1\big)$, and substituting in $(a)$
$$ \HH_1 \XX_1-\BB\HH_2^{-1}\big(\w\,\v^{\top}+\BB^{\top}\XX_1\big)=\II_n-\v\,\v^{\top},$$
which implies that 
$$\II_n-\v\,\v^{\top}+\BB\HH_2^{-1}\w\,\v^{\top}=\Ss\XX_1.$$
Now, multiplying by $\Ss^{\#}$ and taking into account that $\Ss^\#\v=\0$, 
$$ \Ss^{\#}\left(\II_n+\BB\HH_2^{-1}\w\,\v^{\top}\right)=\Ss^{\#}\Ss\XX_1=\left(\II_{n}-\alpha\,\v\,\v^\top \right)\XX_1,$$
%
and applying $(c)$, we get
$$\XX_1+\alpha\, \v\w^{\top}\CC^{\top}=\Ss^{\#}\left(\II_n+\BB\HH_2^{-1}\w\,\v^{\top}\right).$$
Bearing in mind the expression for $\CC^\top$ obtained above,
$$\XX_1-\alpha\, \v\w^{\top}\HH_2^{-1}\big(\w\,\v^{\top}+\BB^{\top}\XX_1\big)=\Ss^{\#}\left(\II_n+\BB\HH_2^{-1}\w\,\v^{\top}\right).$$

Moreover, defining as $\LL=\II_n-\alpha\,\v\w^{\top}\HH_2^{-1}\BB^{\top}$, we get that 
$$ 
\LL\XX_1= \Ss^{\#}\big(\II_n+
\BB\HH_2^{-1}\w\,\v^{\top}\big)+\alpha\, \v\w^{\top}\HH_2^{-1}\w\,\v^{\top}.$$
We next show that $\LL$ is invertible and moreover
$$\LL^{-1}=\II_n+\v\w^{\top}\HH_2^{-1}\BB^{\top}.$$
To prove that, notice that from the identity $\HH_2^{-1}\BB^\top \v=-\w$ we have 
$$\begin{array}{rl}
\LL\Big(\II_n+\v\w^{\top}\HH_2^{-1}\BB^{\top}\Big)=&\hspace{-.25cm}\Big(\II_n-\alpha\,\v\w^{\top}\HH_2^{-1}\BB^{\top}\Big)\Big(\II_n+\v\w^{\top}\HH_2^{-1}\BB^{\top}\Big)\\[1ex]
=&\hspace{-.25cm} \II_n-\alpha\,\v\w^{\top}\HH_2^{-1}\BB^{\top}+\v\w^{\top}\HH_2^{-1}\BB^{\top}-\alpha\,\v\w^{\top}\HH_2^{-1}\BB^{\top}\v\w^{\top}\HH_2^{-1}\BB^{\top}\\[1ex]
=&\hspace{-.25cm} \II_n-(\alpha-1-\alpha\,\w^{\top}\w)\v\w^{\top}\HH_2^{-1}\BB^{\top}=\II_n.\end{array}$$
As a consequence we obtain,
$$ 
\XX_1=\big(\II_n+\v\w^{\top}\HH_2^{-1}\BB^{\top}\big)\Ss^{\#}\big(\II_n+\BB\HH_2^{-1}\w\,\v^{\top}\big)+\v\w^{\top}\HH_2^{-1}\w\,\v^{\top},$$
since 
$$\big(\II_n+\v\w^{\top}\HH_2^{-1}\BB^{\top}\big)\v=\v-\v\,\w^{\top}\w=\alpha^{-1}\v,$$
and moreover that 
$$\begin{array}{rl}
\CC^{\top}=&\hspace{-.25cm}-\HH_2^{-1} \w\,\v^{\top}-\HH_2^{-1} \BB^{\top}\XX_1
\\[1ex]
=&\hspace{-.25cm}-\HH_2^{-1} \w\,\v^{\top}-\big(\II_m-\w\w^{\top}\big)\HH_2^{-1}\BB^{\top}\Ss^{\#}\big(\II_n+\BB\HH_2^{-1}\w\,\v^{\top}\big)+\w\w^{\top}\HH_2^{-1}\w\,\v^{\top}
\\[1ex]
=&\hspace{-.25cm}-\big(\II_m-\w\w^{\top}\big)\HH_2^{-1} \w\,\v^{\top}-\big(\II_m-\w\w^{\top}\big)\HH_2^{-1}\BB^{\top}\Ss^{\#}\big(\II_n+\BB\HH_2^{-1}\w\,\v^{\top}\big)
\\[1ex]
=&\hspace{-.25cm}-\big(\II_m-\w\w^{\top}\big)\HH_2^{-1} \Big[\w\,\v^{\top}+\BB^{\top}\Ss^{\#}\big(\II_n+\BB\HH_2^{-1}\w\,\v^{\top}\big)\Big]
\end{array}$$
 
Finally, from (d), we get that  
$$\begin{array}{rl}
\XX_2=&\hspace{-.25cm}\HH_2^{-1}(\II_m-\w\w^\top)+\HH_2^{-1}\BB^\top\big[\v\w^{\top}+\big(\II_n+\v\w^{\top}\HH_2^{-1}\BB^\top\big)\Ss^{\#}\BB\big]\HH_2^{-1} \big(\II_m-\w\w^{\top}\big)\\[1ex]
=&\hspace{-.25cm}\HH_2^{-1}(\II_m-\w\w^\top)+\big[-\w\,\w^{\top}+\big(\II_m-\w\,\w^{\top}\big)\HH_2^{-1}\BB^\top\Ss^{\#}\BB\big]\HH_2^{-1} \big(\II_m-\w\w^{\top}\big)\\[1ex]
=&\hspace{-.25cm}\HH_2^{-1}(\II_m-\w\w^\top)-\w\,\w^{\top}\HH_2^{-1} \big(\II_m-\w\w^{\top}\big)\\[1ex]
+&\hspace{-.25cm}\big(\II_m-\w\,\w^{\top}\big)\HH_2^{-1}\BB^\top\Ss^{\#}\BB\HH_2^{-1} \big(\II_m-\w\w^{\top}\big)\\[1ex]
=&\hspace{-.25cm}(\II_m-\w\,\w^{\top})\HH_2^{-1}\Big[\HH_2+\BB^\top\Ss^{\#}\BB\Big] \HH_2^{-1}\big(\II_m-\w\w^{\top}\big).\hspace{3cm}\qed \end{array}$$

Observe that if $\v=\w=\0$, then $\XX_1=\Ss^\#=\Ss^{-1}$ and the above formula for $\AA^\#$ becomes \eqref{Schur:inv}, so the preceeding Theorem is also true for the nonsingular case. 
In addition, as a by--product we obtain the following result, see  \cite[Lemma 1]{CEGM16}.

 \begin{corollary} 
\label{lemmap}
Let $\HH \in \mathcal{M}_{n\times n}(\RR)$ be a  symmetric matrix, $\s\in \mathcal{M}_{n\times 1}(\RR)$, $\alpha\not=0$ and assume that  $\AA=\left(\begin{array}{cc}
\HH   & \s\\
\s^{\top} & \alpha
\end{array}
\right)$ is singular and has $0$ as a simple eigenvalue.  If  ${\sf S}={\sf H}-\alpha^{-1}\s\s^{\top}$   then  
$$\AA^\#=\left(\begin{array}{cc}
\MM\Ss^\#\MM +\alpha^{-1}w^2\v\v^{\top}& -\alpha^{-1} \big(1-w^2\big)\big(w\v+\MM\Ss^\#\s\big)\\[1ex]
-\alpha^{-1} \big(1-w^2\big) \big(w\v^\top+\s^\top\Ss^\#\MM^\top\big)& \alpha^{-2}(1-w^2)^2\Big[\alpha+\s^\top\Ss^{\#}\s\Big]
\end{array}
\right),$$
where $\u^\top=(\v^\top,w)$ is an unitary $0$--eigenvector of $\AA$ and $\MM= \II_n+\alpha^{-1}w\v\s^\top$.\end{corollary}

Notice that the above corollary is also true when $\AA$ is invertible. Moreover, to assure that $0$ is simple at most, it is suffices to assume that $\HH$ is invertible, that is the hypothesis in \cite{CEGM16}.



\section{The group inverse of the Laplacian of two connected networks}
\label{}

In the following, the triple $\Gamma=(V,E,c)$ denotes a finite network, i.e., a connected graph without loops nor multiple edges with vertex set $V$, with cardinality $n$, and edge set $E$, in which each edge $e_{xy}\in E$ has assigned a value $c(x,y)>0$ named conductance. The conductance $c$ is a symmetric function $c: V\times V \rightarrow [0,\infty)$ such that $c(x,x)=0$ for any $x\in V$ and where the vertex $x$ is adjacent to vertex $y$ iff $c(x,y)>0$.
The  Laplacian of the network $\Gamma$
is the endomorphism of $\CV$ that assigns to each $u\in \CV$ the function
$$\L(u)(x)=\sum_{y\in V}c(x,y)\left(u(x)-u(y)\right), \quad x\in V.$$
The Laplacian is a singular elliptic operator on $\CV$ and moreover $\L(u)=0$ iff $u$ is a constant function. Its kernel is the known Laplacian matrix ${\sf L}$. Given a function $q\in \CV$, the Schr\"odinger operator on $\Gamma$ with potential $q$ is the endomorphism of $\CV$ that assigns to each $u\in \CV$ the function $\L_q(u)=\L(u)+qu$. The associated kernel is denoted by ${\sf L}_q$. Besides, given a weight $\omega\in \Omega(V)$, the potential determined by the weight $\omega$ is the function $q_{\omega}=-\frac{1}{\omega}\L(\omega)$. It is well--known that the Schr\"odinger operator $\L_q$ is $(\lambda,\omega)$--elliptic iff $q=q_{\omega}+\lambda$. Moreover, it is  singular iff $\lambda=0$ and in this case is the Laplacian operator. Then, $q=q_{\omega}+\lambda=0$, and
$\L(v)=0$ iff $v=a\omega$, $a\in \mathbb{R}$, that is, the eigenvector ${\sf v}$ associated to the eigenvalue $0$ is proportional to the weight $\omega$. If $\G$ is the orthogonal Green operator associated with $\L_q$, and ${\sf G}$ is its corresponding kernel, ${\sf G}$ is the group inverse of the Laplacian matrix ${\sf L}$ iff $\lambda=0$ and $v$ is proportional to the normalized weight $\omega=\frac{1}{\sqrt{n}}{\sf j}$, where ${\sf j}$ stands for the all-ones vector and ${\sf J}_{n,m}$ stands for the $(n,m)$-all-ones matrix. Observe that for $\lambda>0$ the kernel of $\L_q$ is non-singular, and then its inverse can be computed with the Identities (\ref{Schur:inv}) or (\ref{Schur:inv2}).

Now, we consider two networks $\Gamma_1=(V_1,E_1,c_1)$ and $\Gamma_2=(V_2,E_2,c_2)$, and we connect a subset of vertices $\{x_1,\dots, x_{m_1}\}\subseteq V_1$ with a subset of vertices $\{y_1,\dots, y_{m_2}\} \subseteq V_2$, by means of some new edges of conductances $a_{ij}=a(x_i,y_j)$, where $1\leq i \leq m_1 \leq n_1$, $1\leq j\leq m_2 \leq n_2$. Thus, we denote by $\Gamma=(V,E,c)$ the new network, that has as new vertex set $V=V_1\cup V_2$ ($n=n_1+n_2$), new edge set $E=E_1\cup E_2 \cup \{e_{x_1y_1},\dots,e_{x_{m_1}y_{m_2}}\}$ and  new conductance
$$c'(x,y)=
\left\{\begin{array}{ll}
a_{ij}>0 & \textrm{if} \; \;x=x_i\in V_1,\; y=y_j\in V_2, \,1\leq i \leq m_1,\,1\leq j \leq m_2,\\
c_1(x,y) & \textrm{if}  \; \;x,y\in V_1,\\
c_2(x,y) & \textrm{if}  \; \;x,y\in V_2,\\
0& \textrm{otherwise}.
\end{array}
\right.
 $$

Besides, if $\omega_1(x)$ denotes the weight associated to $\Gamma_1$ and $\omega_2(x)$ denotes the weight associated to $\Gamma_2$, we define $\omega(x)$ the weight associated to $\Gamma$ as follows: 
$\omega(x)=\omega_i(x)/\sqrt{\omega_1(x)^2+\omega_2(x)^2}$ for any $x\in V$, $i=1,2$. 
Observe that the relation $\omega_i(x)/\omega_i(y)=\omega(x)/\omega(y)$ holds for any $x$, $y\in V$, $i=1,2$.

Given $\lambda\geq 0$,  consider $q\in \CV$ and $p_i\in \mathcal{C}(V_i)$, $i=1,2$ and the potentials $q=q_\omega+\lambda$ and $p_i=q_{\omega_i}+\lambda$, $i=1,2$. In the sequel we consider fixed $\lambda\ge 0$ and hence the potentials $q$ and $p_1,p_2$, and we denote by $\G_q$ and ${\sf G}_q$ the Green operator and the Green function corresponding to $\L_q$, and by $\G^i_{p_i}$ and ${\sf G}^i_{p_i}$ the Green operator and the Green function respectively, corresponding to $\L^i_{p_i}$, $i=1,2$. 

Besides we use the following notation: for any $i=1,\ldots,m_1$, $j=1,\ldots,m_2$ we denote by $\rho_{ij}=\sqrt{a_{ij}\omega(x_i)\omega(y_i)}$.

\begin{prop}
For any $u\in \mathcal{C}(V)$ it is satisfied
$$\begin{array}{rll}
\L_qu(x_i)=&\hspace{-.25cm}\L^1_{p_1} u(x_i)-\sum\limits_{j=1}^{m_2} a_{ij} u(x_i)+\sum\limits_{j=1}^{m_2}\dfrac{\rho^2_{ij}}{\omega^2(x_i)}u(x_i),&\quad i=1,\dots,m_1,\\[2ex]
\L_qu(y_j)=&\hspace{-.25cm}\L^2_{p_2} u(y_j)-\sum\limits_{i=1}^{m_1} a_{ij} u(y_j)+\sum\limits_{i=1}^{m_1}\dfrac{\rho^2_{ij}}{\omega^2(y_j)}u(y_j),&\quad j=1,\dots,m_2,\\[2ex]
\L_qu(x)=&\hspace{-.25cm}\L^i_{p_i} u(x),  \quad i=1,2,\quad \hbox{otherwise}.
\end{array}$$
\end{prop}

In terms of matrices, the relation between the kernels of the Schr\"odinger operators of $\Gamma_1$, $\Gamma_2$ and $\Gamma'$ is given by
$$ {\sf L}'_{p}=
\left(\begin{array}{cc}
{\sf H}_1 & {\sf B}\\
 {\sf B}^{\intercal} & {\sf H}_2
\end{array}
\right),$$
where for $i=1,2$ each matrix ${\sf H}_i$ is the sum of the original Schr\"odinger matrix ${\sf L}_{p}^i$ plus a dia\-gonal matrix ${\sf D}_i$, such that 
\begin{eqnarray*}
({\sf D}_1)_{kk}=\sum_{j=1}^{m_2}\dfrac{\omega(y_j)}{\omega(x_k)} a_{kj},& \textrm{for any $k=1,\dots,m_1$, }\\
({\sf D}_2)_{kk}=\sum_{i=1}^{m_1}\dfrac{\omega(x_k)}{\omega(y_j)} a_{kj},& \textrm{for any $k=1,\dots,m_2$. }
\end{eqnarray*}
Moreover, if  ${\sf e}_i$ are the vectors of the standard basis, ${\sf B}$ is the matrix whose columns are $\displaystyle-\sum_{i=1}^{m_1} a_{ij}{\sf e}_i$, for $j=1,\dots,m_2$.

In particular, for Laplacian matrices, $\lambda=0$ and the weight is constant, $\omega=(1/\sqrt{n}){\sf j}$,and the relation among ${\sf L}_1$, ${\sf L}_2$ and ${\sf L}'$ is given by
$$ {\sf L}'=
\left(\begin{array}{cc}
{\sf H}_1 & {\sf B}\\
 {\sf B}^{\intercal} & {\sf H}_2
\end{array}
\right),$$
where for $i=1,2$ each matrix ${\sf H}_i$ is the sum of the Laplacian matrix ${\sf L}_i$ plus a dia\-gonal matrix ${\sf D}_i$, such that for any $k=1,\dots,m_1$, $({\sf D}_1)_{kk}=\sum_{j=1}^{m_2} a_{kj}$ and for any $k=1,\dots,m_2$, $({\sf D}_2)_{kk}=\sum_{i=1}^{m_1} a_{kj}$.
And ${\sf B}$ is again the matrix whose columns are $\displaystyle-\sum_{i=1}^{m_1} a_{ij}{\sf e}_i$, for $j=1,\dots,m_2$.

Our main result, relating the group inverses of the Laplacian matrices of the two original networks, with the group inverse of the Laplacian matrix of the new network is given in the following result:
\begin{theorem} \label{maintheorem}
 The group inverse of ${\sf L}'$ is given by
 $$({\sf L}')^{\dag}
=
\left(
\begin{array}{cc}
\XX_1
& \CC\\[2ex]
\CC^{\top}
&\XX_2
\end{array}
\right),
$$
where 
\begin{eqnarray*}
\XX_1&\hspace{-.25cm}=&\hspace{-.25cm} \frac{1}{n_1^2n_2^2}
\left[(n_1n_2\II_n+{\sf J_{n_1n_2}}\HH_2^{-1}\BB^{\top})\MM 
(n_1n_2{\sf I_{n_1}}+\BB\HH_2^{-1}{\sf J_{n_2n_1}})+{\sf J_{n_1n_2}}\HH_2^{-1}{\sf J_{n_2n_1}}\right],\\[1.5ex]
\XX_2&\hspace{-.25cm}=& \hspace{-.25cm} \left({\sf I_{n_2}}-\frac{1}{n^2_2}{\sf J_{n_2}}\right)\HH_2^{-1}(\HH_2+\BB^{\top}\MM\BB)\HH_2^{-1}\left({\sf I_{n_2}}-\frac{1}{n^2_2}{\sf J_{n_2}}\right),\\[1.5ex]
\CC&\hspace{-.25cm}= &\hspace{-.25cm}-\left[\frac{1}{n_1n_2}{\sf J_{n_1n_2}}+\left({\sf I_{n_1}}+\frac{1}{n_1n_2}{\sf J_{n_2n_1}}\HH_2^{-1}\BB^{\top}\right)\MM\BB\right]\HH_2^{-1}\left( {\sf I_{n_2}}-\frac{1}{n^2_2}{\sf J_{n_2}}\right),\\[1.5ex]
\end{eqnarray*}
the matrix ${\sf M}=(\HH_1-\BB\HH_2^{-1}\BB^{\top})^{\dag}$ and ${\sf H}_2={\sf L_2}
+\hbox{diag}(\sum_{i=1}^{m_1} a_{1j},\dots,\sum_{i=1}^{m_1} a_{n_2j})$.
\end{theorem}

\begin{proof}
The result is straightforward from Theorem \ref{lemmap}, taking into account that
${\sf v}=(1/\sqrt{n}){\sf j}_{n_1}$, ${\sf w}=(1/\sqrt{n}){\sf j}_{n_2}$.
\end{proof}

\section{Acknowledgement}
This work has been partially supported by the Mathematics Department of the Universitat Polit\`ecnica de Catalunya (BarcelonaTech).







\end{document}